\documentclass[12pt]{amsart}

\usepackage[utf8]{inputenc}
\usepackage{txfonts}
\usepackage{amsmath,amssymb,amsthm}
\usepackage{graphicx}
\usepackage{tikz}
\usetikzlibrary{trees,positioning,calc}
\usepackage{array}
\usepackage{caption}
\usepackage[colorlinks=true,
            linkcolor=red,
            citecolor=green,
            urlcolor=blue,
            pagebackref=true,
            final,
            hyperindex]{hyperref}

\renewcommand*{\backrefalt}[4]{%
  \ifcase #1
  \relax
  \else
  \space{\footnotesize #2}%
  \fi
}

\allowdisplaybreaks[4]
\numberwithin{equation}{section}

\newcolumntype{C}[1]{>{\centering\arraybackslash}p{#1}}

\providecommand{\mlabel}[1]{\label{#1}}
\providecommand{\mcite}[1]{\cite{#1}}
\providecommand{\mref}[1]{\ref{#1}}
\providecommand{\meqref}[1]{\eqref{#1}}
\providecommand{\mbibitem}[1]{\bibitem{#1}}

\newif\ifshowlab
\showlabfalse     

\ifshowlab
  \renewcommand{\mlabel}[1]{\label{#1}{\hfill\quad\mbox{\normalfont\bfseries (#1)}}}

  \renewcommand{\mcite}[1]{\cite{#1}}

  \renewcommand{\mref}[1]{\ref{#1}\ifmmode\text{\normalfont\bfseries\ (#1)}\else{\normalfont\bfseries\ (#1)}\fi}
  \renewcommand{\meqref}[1]{\eqref{#1}\ifmmode\text{\normalfont\bfseries\ (#1)}\else{\normalfont\bfseries\ (#1)}\fi}

  \renewcommand{\mbibitem}[1]{\bibitem{#1}}
\else
  \renewcommand{\mlabel}[1]{\label{#1}}
  \renewcommand{\mcite}[1]{\cite{#1}}
  \renewcommand{\mref}[1]{\ref{#1}}
  \renewcommand{\meqref}[1]{\eqref{#1}}
  \renewcommand{\mbibitem}[1]{\bibitem{#1}}
\fi

\newcommand{\bk}{\mathbf{k}}
\newcommand{\Y}{\mathcal{Y}}

\newcommand{\id}{\mathrm{id}}

\newcommand{\DeltaLR}{\Delta_{\rm LR}^{\epsilon}}
\newcommand{\DeltaRT}{\Delta_{\rm RT}^{\epsilon}}

\newcommand{\funit}{\mathbf{1}}

\theoremstyle{plain}
\newtheorem{theorem}{Theorem}[section]
\newtheorem{lemma}[theorem]{Lemma}
\newtheorem{coro}[theorem]{Corollary}
\newtheorem{prop}[theorem]{Proposition}

\theoremstyle{definition}
\newtheorem{definition}[theorem]{Definition}

\newtheorem{exam}[theorem]{Example}
\newtheorem{remark}[theorem]{Remark}

\newcolumntype{M}[1]{>{\centering\arraybackslash}m{#1}}

\newcommand\yy[2]{%
  \begin{scope}[scale=1/2^#1]
    \draw (#2,0)--+(1/2,-1/2)--+(1,0);
  \end{scope}
}

\newcommand\YY[2][]{%
  \tikz[line width=0.18ex,scale=0.75,baseline=-3ex,inner sep=1pt,#1]{%
    \draw (0,0)--+(1/2,-1/2)--+(1,0)
          (1/2,-1/2)--+(0,-1/2);
    #2
  }%
}

\begin{document}

\title[Infinitesimal Bialgebra on Planar Binary Trees]{Infinitesimal Bialgebra on Planar Binary Trees}
\author{Yong Yu}
\address{School of Mathematics and Statistics, Lanzhou University, Lanzhou, 730000, China}
\email{yuyong\_lzu@yeah.net}

\author{Xing Gao$^{*}$}
\address{School of Mathematics and Statistics, Lanzhou University, Lanzhou, 730000, China; Gansu Provincial Research Center for Basic Disciplines of Mathematics and Statistics, Lanzhou, 730070, China}
\email{gaoxing@lzu.edu.cn}

\thanks{$^{*}$ Corresponding author.}

\date{\today}
\begin{abstract}
We construct a weight-zero infinitesimal bialgebra structure on the
$\bk$-module spanned by planar binary trees, using the 
under product $\backslash$ of Aguiar--Sottileand a root-recursive coproduct $\DeltaLR$.
We prove that $\DeltaLR$ is coassociative and satisfies the infinitesimal
derivation rule with respect to $\backslash$, hence gives a unitary
infinitesimal bialgebra distinct from the usual Loday--Ronco Hopf structure.
We also obtain an elementary vertex-cut formula, establish freeness properties for unitary $(\backslash,\vee)$-algebras and
unitary infinitesimal $(\backslash,\vee)$-bialgebras, and identify
the construction with the infinitesimal coproduct transported from planar
rooted forests.
\end{abstract}



\maketitle
\vspace{-1.2cm}
\tableofcontents
\vspace{-1.3cm}

\section{Introduction}\mlabel{sec:introduction}

\subsection{Motivation: infinitesimal compatibility on planar binary trees}

Infinitesimal bialgebras were introduced by Joni and Rota as
algebra--coalgebra structures in which the coproduct is a derivation with
respect to the product~\mcite{Joni1979}.  In weight zero, this compatibility is
expressed by
\[
\Delta(ab)=a\cdot \Delta(b)+\Delta(a)\cdot b.
\]
Aguiar further developed infinitesimal Hopf algebras and emphasized that they
are not merely formal variants of ordinary Hopf algebras: they are related to
associative Yang--Baxter equations, pre-Lie structures, dendriform structures,
and combinatorial Hopf algebras~\mcite{Aguiar2000}.  The present paper takes
this infinitesimal point of view in a concrete tree-combinatorial setting.
Instead of asking for a Hopf coproduct on planar binary trees, we ask for a
coproduct which is compatible with the associative operation
$\backslash$ in the infinitesimal sense.

Planar binary trees already carry rich Hopf-theoretic structures.  The
Loday--Ronco Hopf algebra~\mcite{Loday1998} relates them to associahedra,
bracketings, and the Tamari order
\mcite{Stasheff1963a,Stasheff1963b,Tamari1962}.  The Aguiar--Sottile analysis
of this Hopf algebra describes a distinguished associative operation
$\backslash$, progressive trees, and a rooted-forest/binary-tree correspondence
\mcite{Aguiar2006}.  On the other hand, planar rooted forests form a classical
source of combinatorial Hopf and infinitesimal bialgebras, including the
Connes--Kreimer, Grossman--Larson, and related rooted-tree constructions
\mcite{Connes1998,Connes2000,Grossman1989,Kreimer1998,Schmitt1994}.  Weighted
infinitesimal bialgebras on rooted forests and their freeness properties were
constructed in \mcite{zhang2019,Zhang2022a}. The pre-Lie and operadic aspects of rooted trees are part of the background
developed in \mcite{Chapoton2001,Chapoton2004,Chapoton2007}.  Further
Hopf-algebraic developments for rooted trees and decorated rooted trees were
obtained by Foissy and Holtkamp
\mcite{Foissy2001,Foissy2002a,Foissy2002b,Foissy2007,Holtkamp2003,Holtkamp2006}.
Related algebraic structures on or near planar binary trees include
dialgebras, trialgebras, generalized bialgebras, magmatic bialgebras, and Hopf
operads \mcite{Holtkamp2008,Loday2004,Loday2008}.  Planar binary trees also
appear in QED and renormalization-inspired Hopf algebras, as well as in the
Hopf-theoretic study of bi-leveled trees and multiplihedra
\mcite{Brouder2003,Brouder2010,Forcey2009}.  Noncommutative symmetric
functions, free quasi-symmetric functions, and the Malvenuto--Reutenauer Hopf
algebra provide another structural background for the Loday--Ronco setting
\mcite{Duchamp2002,Gelfand1995,Hazewinkel2001,Malvenuto1995,Ronco2002}.
Operated-algebra methods, including free operated algebras, Rota--Baxter type
constructions, rooted-forest cocycles, and cocycle $\vee$-Hopf algebras, give
a parallel framework for the constructions considered here
\mcite{Bokut2010,Cartier1972,Gao2016,Guo2009,Zhang2020,Zhang2022b}.

These two developments suggest a natural problem.   Since forest concatenation has a
weight-zero infinitesimal coproduct, and since the Aguiar--Sottile
correspondence sends forest concatenation to the product $\backslash$ on
planar binary trees, one expects a corresponding infinitesimal coproduct on
planar binary trees.  However, this coproduct should not be confused with the
Loday--Ronco Hopf coproduct, nor with the admissible-cut coproducts appearing
in the Connes--Kreimer or Loday--Ronco Hopf settings.  The purpose of this
paper is to construct such an infinitesimal coproduct directly on planar binary
trees and then to explain its universal and functorial meaning.

\subsection{The guiding problem and the choice of operations}

The operation used throughout the paper is the associative operation
$\backslash$ from the Aguiar--Sottile framework~\mcite{Aguiar2006}, not the original
Loday--Ronco Hopf product in the fundamental tree basis.  For this operation,
progressive planar binary trees are the free associative generators.  The
problem addressed here can therefore be stated as follows:
\begin{center}
\fbox{\begin{minipage}{0.92\textwidth}
\textbf{Guiding problem.}
Construct and characterize a coproduct $\DeltaLR$ on the $\bk$-module spanned
by planar binary trees such that
\[
\DeltaLR(S\backslash T)
=
S\backslash \DeltaLR(T)+\DeltaLR(S)\backslash T,
\]
and such that the construction is compatible with the binary grafting operation
$\vee$ and with the rooted-forest infinitesimal coproduct.
\end{minipage}}
\end{center}

The binary grafting operation $\vee$ is also essential.  It gives the root
decomposition $T=T_l\vee T_r$ of every non-trivial planar binary tree.  The
coproduct constructed in \meqref{eq:yucheng} is root-recursive with respect to
this decomposition, while the infinitesimal bialgebra property is proved with
respect to the product $\backslash$ in Lemma~\ref{lem:derivation} and
Theorem~\ref{thm:inf-bialgebra}.  Thus the paper studies the interaction of two
operations: $\vee$ controls the recursive construction, and $\backslash$
controls the infinitesimal compatibility.

\subsection{Position among nearby tree bialgebras}

The construction in this paper sits close to several familiar Hopf and
infinitesimal frameworks, but it is not identical to any of them.  The table
below records the distinctions that will be used throughout the paper.

\begin{center}
\renewcommand{\tabcolsep}{4pt}
\footnotesize
\begin{tabular}{|M{0.25\textwidth}|M{0.22\textwidth}|M{0.25\textwidth}|M{0.20\textwidth}|}
\hline
\textbf{Framework} & \textbf{Product} & \textbf{Coproduct / cuts} & \textbf{Role in this paper} \\
\hline
\textbf{Planar rooted-forest Hopf algebras}
& Forest product in the Hopf setting
& Hopf coproduct described by admissible cuts
& Hopf-type contrast; not the coproduct studied here \\
\hline
\textbf{Planar rooted-forest infinitesimal bialgebras}
& Forest concatenation
& Recursive infinitesimal coproduct governed by $B^+$
& Source of the transported infinitesimal structure in Theorem~\ref{thm:bridge} \\
\hline
\textbf{Present planar binary-tree infinitesimal bialgebra}
& Under product $\backslash$ together with the grafting operation $\vee$
& Root-recursive coproduct $\DeltaLR$ and elementary vertex cuts
& Main object constructed in \meqref{eq:yucheng} and Theorem~\ref{thm:inf-bialgebra} \\
\hline
\end{tabular}
\captionof{table}{Three nearby tree bialgebra frameworks.}
\mlabel{tab:comparison-frameworks}
\end{center}
Table~\ref{tab:comparison-frameworks} is meant to prevent two possible
misreadings.  First, $\DeltaLR$ is not the Loday--Ronco Hopf coproduct in a
new basis.  It is an infinitesimal coproduct for the under product
$\backslash$.  Second, although Theorem~\ref{thm:bridge} proves that the
rooted-forest infinitesimal coproduct is transported to $\DeltaLR$ by the
map $\Theta$, this bridge is an a posteriori identification.  The coproduct
$\DeltaLR$ is constructed internally on planar binary trees in
\meqref{eq:yucheng}, and its main properties are proved before the bridge with
rooted forests is established.

\subsection{Main results and their significance}

The main contributions of the paper are the following.

\begin{enumerate}
\item \textbf{A root-recursive infinitesimal coproduct.}
We define a linear map
$\DeltaLR:H_{\rm LR}^{\epsilon}\to H_{\rm LR}^{\epsilon}\otimes
H_{\rm LR}^{\epsilon}$ recursively by
\[
\DeltaLR(T_l\vee T_r)
=
T_l\otimes T_r
+(\id\otimes\vee)(\DeltaLR(T_l)\otimes T_r)
+(\vee\otimes\id)(T_l\otimes \DeltaLR(T_r)),
\]
with $\DeltaLR(|)=0$; see \meqref{eq:yucheng}. Proposition~\ref{lem:coassociative} proves
coassociativity, while  Lemma~\ref{lem:derivation}
shows that this coproduct satisfies the weight-zero infinitesimal derivation
rule with respect to $\backslash$.  Consequently,
$(H_{\rm LR}^{\epsilon},\backslash,|,\DeltaLR)$ is a unitary infinitesimal bialgebra of weight
zero by Theorem~\ref{thm:inf-bialgebra}.  This gives a natural infinitesimal
bialgebra structure on planar binary trees which is different from the usual
Hopf structures on the same combinatorial objects.

\item \textbf{An elementary vertex-cut formula.}
For each internal vertex $e$ of a planar binary tree $T$, we define the
elementary vertex cut $e$ and its associated pair
$(P^{e}(T),R^{e}(T))$ in Definition~\ref{def:elementary-vertex-cut}.  The
main formula is
\[
\DeltaLR(T)=\sum_{e\in {\rm Int}(T)} P^{e}(T)\otimes R^{e}(T),
\]
proved in Theorem~\ref{thm:cut-formula}.  Its significance is that the
infinitesimal coproduct is controlled by single-vertex cuts rather than by
Hopf-type admissible cuts.  This gives a direct combinatorial interpretation
of every summand of $\DeltaLR(T)$.

\item \textbf{One-sided grafting identities and comb examples.}
The grafting operation $\vee$ gives two unary operators
$B^{r}(T)=T\vee |$ and $B^{\ell}(T)=|\vee T$.  Proposition~\ref{prop:BellBr-coproduct}
shows that these operators satisfy recursive identities with respect to
$\DeltaLR$.  Corollary~\ref{cor:comb-coproduct} then gives closed formulas for
the coproducts of the left and right comb families.  These formulas provide
explicit test cases for the recursive definition and illustrate how the
infinitesimal structure behaves on basic tree families.

\item \textbf{A universal property for the pair $({\backslash},\vee)$.}
We introduce unitary $(\backslash,\vee)$-algebras and prove that
$(H_{\rm LR}^{\epsilon},\backslash,|,\vee)$ is the free such algebra on the
empty set in Theorem~\ref{thm:free-vee-empty}.  We then introduce unitary
infinitesimal $(\backslash,\vee)$-bialgebras in
Definition~\ref{def:inf-vee-bialgebra}.  Proposition~\ref{prop:HLR-inf-vee}
shows that the planar binary tree object constructed here is an object of this
category, and Theorem~\ref{thm:free-inf-vee} proves that it is free.  This shows that $\Delta_{\rm LR}^{\epsilon}$ is not an auxiliary coproduct
introduced only for the present construction; rather, it is determined by the
unit, the under product $\backslash$, the grafting operation $\vee$, and the
recursive infinitesimal compatibility. 

\item \textbf{A bridge with planar rooted forests.}
The final part of the paper compares the internal construction with the known
infinitesimal bialgebra on planar rooted forests.  Proposition~\ref{prop:uniqbl}
first proves a uniqueness criterion for $\DeltaLR$ in terms of $\DeltaLR(|)=0$,
the $B^{r}$-recursion, and the infinitesimal derivation rule.  Proposition~\ref{prop:theta-alg-iso}
constructs a unitary algebra isomorphism
$\Theta:(\bk\mathcal F,\cdot,\funit)
\longrightarrow
(H_{\rm LR}^{\epsilon},\backslash,|)$ which sends forest
concatenation to $\backslash$ and $B^+$ to $B^{r}$.  Theorem~\ref{thm:bridge}
then proves
\[
(\Theta\otimes\Theta)\DeltaRT(F)=\DeltaLR(\Theta(F)),
\quad \forall \,F\in \bk\mathcal F.
\]
Thus the planar binary tree coproduct is exactly the transported
rooted-forest infinitesimal coproduct, although its construction and proof are
carried out intrinsically on planar binary trees first.
\end{enumerate}

Taken together, these results clarify the infinitesimal counterpart of the
Aguiar--Sottile rooted-forest/binary-tree correspondence.  They also separate
three different notions which are often close in tree combinatorics: Hopf
admissible cuts, rooted-forest infinitesimal cocycles, and the elementary
vertex-cut coproduct on planar binary trees.

\medskip
\noindent{\bf Convention}
Throughout this paper, $\bk$ is a unitary commutative ring.  All modules,
algebras, coalgebras, bialgebras, tensor products, and linear maps are taken
over $\bk$.  By  an algebra we mean an associative algebra, not necessarily unitary.  For any set $X$, $\bk X$ denotes the free
$\bk$-module with basis $X$.  We use Sweedler's notation for the coproduct:
\[
\Delta(X)
=
\sum_{(X)}X_{(1)}\otimes X_{(2)}.
\]

\section{The infinitesimal bialgebra of planar binary trees}\mlabel{sec:bialgebra}

This section constructs the infinitesimal bialgebra structure on the
$\bk$-module spanned by planar binary trees.  We begin by fixing the basic
tree operations used throughout the construction. 

\subsection{Planar binary trees and  the  grafting operation}\label{subsec:planartree}
We first recall the operations $\vee$ and $\backslash$ on planar binary
trees and the corresponding bimodule structure on $H_{\rm LR}^{\epsilon} \otimes H_{\rm LR}^{\epsilon}$.
\begin{definition}\mcite{Aguiar2000,Zhang2022a}
A  {\bf unitary infinitesimal bialgebra (of weight zero)} is a quadruple
$(A,\cdot,1_A,\Delta)$, where $(A,\cdot,1_A)$ is a unitary associative algebra
and
\[
\Delta:A\longrightarrow A\otimes A
\]
is a coassociative coproduct, not necessarily counitary, such that
\begin{equation}
\Delta(ab)
=
a\cdot \Delta(b)
+
\Delta(a)\cdot b,
\quad
\forall\, a,b\in A.
\mlabel{eq:infinitesimal-compatibility}
\end{equation}
\end{definition}

A \textbf{planar binary tree}~\mcite{Loday1998} is a finite rooted planar tree in which every
internal vertex has exactly two ordered incoming branches, called its left and
right branches.  The trivial binary tree, denoted by $|$, has no internal vertices.  Let
$\Y$ be the set of planar binary trees.  For $T\in\Y$, we denote by
${\rm Int}(T)$ the set of internal vertices of $T$; in particular, ${\rm Int}(|)=\emptyset$.  For $n\ge 0$, let
$\Y_n$ denote the set of planar binary trees satisfying
$|{\rm Int}(T)|=n$; thus $\Y_0=\{|\,\}$. For example,
\begin{align*}
\Y_1=
\bigg\{ \YY{} \bigg\},
\quad
\Y_2=
\bigg\{\YY{\yy10}, \, \YY{\yy11}\bigg\},
\quad
\Y_3=
\bigg\{ \YY{\yy10\yy20}, \, \YY{\yy11\yy23}, \,
\YY{\yy20\yy23}, \, \YY{\yy10\yy21}, \, \YY{\yy11\yy22} \bigg\}.
\end{align*}
Set
\[
H_{\rm LR}^{\epsilon}:=\bk\Y=\bigoplus_{n\ge 0}\bk\Y_n.
\]

\begin{definition}\mcite{Loday1998,Zhang2020}
The \textbf{binary grafting operation} is the bi-linear map
\[
\vee:H_{\rm LR}^{\epsilon}\otimes H_{\rm LR}^{\epsilon}
\longrightarrow H_{\rm LR}^{\epsilon},
\quad
T_l\otimes T_r\longmapsto T_l\vee T_r,
\]
where $T_l\vee T_r$ is obtained by creating a new root and attaching $T_l$
and $T_r$ as its left and right branches, respectively.
\mlabel{de:graft}
\end{definition}

For example,
\[
|\vee \YY{}=\YY{\yy11},
\quad
\YY{}\vee \YY{}=\YY{\yy20\yy23}.
\]
Every non-trivial planar binary tree $T\in\Y$, $T\neq |$, admits a unique
root decomposition
$T=T_l\vee T_r,$
where $T_l$ and $T_r$ are called respectively the left and right branches of the root.
For instance,
\[
\YY{\yy10\yy20}=\YY{\yy10}\vee |.
\]
The trees
\[
|,
\quad
\YY{},
\quad
\YY{\yy10},
\quad
\YY{\yy10\yy20},
\quad
\YY{\yy10\yy20\yy30},
\quad \ldots
\]
are called the \textbf{right comb trees}~\mcite{Aguiar2006}.  Denote by
$F_n^{r}$ the right comb tree with $n$ internal vertices:
\begin{equation}
F_0^{r}:=|,
\quad
F_n^{r}:=F_{n-1}^{r}\vee |,
\quad \forall \, n\ge 1.
\mlabel{eq:right-comb}
\end{equation}
Similarly, the left comb trees are given by
\begin{equation}
F_0^{\ell}:=|,
\quad
F_n^{\ell}:=|\vee F_{n-1}^{\ell},
\quad \forall \, n\ge 1.
\mlabel{eq:left-comb}
\end{equation}

\subsection{Recursive definition of the infinitesimal coproduct}
\mlabel{subsec:recursive}

We now define a coproduct on $H_{\rm LR}^{\epsilon}$ by recursion on the root
decomposition $T=T_l\vee T_r$.  This construction only uses the grafting
operation $\vee$ at this stage.  Its coassociativity will be proved first; the
under product $\backslash$ will then be introduced in order to formulate and
prove the infinitesimal compatibility relation.

Define the {\bf infinitesimal coproduct} to be the linear map
\[
\Delta_{\rm LR}^{\epsilon}:H_{\rm LR}^{\epsilon}
\longrightarrow
H_{\rm LR}^{\epsilon}\otimes H_{\rm LR}^{\epsilon}
\]
by setting
\begin{equation}
\Delta_{\rm LR}^{\epsilon}(T):=
\left\{
\begin{array}{c@{\quad}l}
0,
& \text{if} \,T=|,
\\[1mm]
\begin{aligned}
&T_l\otimes T_r
+(\id\otimes \vee)\bigl(\Delta_{\rm LR}^{\epsilon}(T_l)\otimes T_r\bigr)
\\
&\quad
+(\vee\otimes \id)\bigl(T_l\otimes \Delta_{\rm LR}^{\epsilon}(T_r)\bigr)
\end{aligned}
& \text{if} \, T=T_l\vee T_r\neq |.
\end{array}
\right.
\mlabel{eq:yucheng}
\end{equation}
Equivalently, for $T=T_l\vee T_r\neq |$,
\begin{equation}
\Delta_{\rm LR}^{\epsilon}(T)
:=
T_l\otimes T_r
+
\sum_{(T_l)}
T_{l,(1)}\otimes \bigl(T_{l,(2)}\vee T_r\bigr)
+
\sum_{(T_r)}
\bigl(T_l\vee T_{r,(1)}\bigr)\otimes T_{r,(2)}.
\mlabel{eq:Delta-expanded}
\end{equation}

\begin{exam}
For some low-degree planar binary trees, we have
\begin{align*}
\Delta_{\rm LR}^{\epsilon}(\YY{\yy10\yy20})
&=
\YY{}\otimes \YY{}
+|\otimes \YY{\yy10}
+\YY{\yy10}\otimes |
,
\\[1mm]
\Delta_{\rm LR}^{\epsilon}(\YY{\yy20\yy23})
&=
\YY{}\otimes \YY{}
+|\otimes \YY{\yy11}
+\YY{\yy10}\otimes |.
\end{align*}
For comparison, let $\Delta$ denote the Loday--Ronco
Hopf coproduct on undecorated planar binary trees~\mcite{Loday1998}.  For the same two trees,
one has
\begin{align*}
\Delta(\YY{\yy10\yy20})
&=
\YY{\yy10}\otimes \YY{}
+
\YY{}\otimes \YY{\yy10}
+
|\otimes \YY{\yy10\yy20}
+\YY{\yy10\yy20}\otimes |,
\\[1mm]
\Delta(\YY{\yy20\yy23})
&=
\YY{\yy10}\otimes \YY{}
+
\YY{\yy11}\otimes \YY{}
+
\YY{}\otimes \YY{\yy10}
+
\YY{}\otimes \YY{\yy11}
+\YY{\yy20\yy23}\otimes |
+
|\otimes \YY{\yy20\yy23}.
\end{align*}
\mlabel{exam:coproduct-basic}
\end{exam}

\begin{prop}
 $\Delta_{\rm LR}^{\epsilon}$  is coassociative:
\[
(\Delta_{\rm LR}^{\epsilon}\otimes \id)
\Delta_{\rm LR}^{\epsilon}
=
(\id\otimes \Delta_{\rm LR}^{\epsilon})
\Delta_{\rm LR}^{\epsilon}.
\]
\mlabel{lem:coassociative}
\end{prop}

\begin{proof}
Let $T \in \Y$. We prove the identity by induction on $|{\rm Int}(T)|\geq 0$. For the initial step of $|{\rm Int}(T)|=0$, we have $T=|$.  Hence
\[
(\Delta_{\rm LR}^{\epsilon}\otimes \id)
\Delta_{\rm LR}^{\epsilon}(|)
=
0
=
(\id\otimes \Delta_{\rm LR}^{\epsilon})
\Delta_{\rm LR}^{\epsilon}(|).
\]
For the induction step of $|{\rm Int}(T)|\ge 1$, 
\begin{align}
(\Delta_{\rm LR}^{\epsilon}\otimes \id)
\Delta_{\rm LR}^{\epsilon}(T)
&\overset{\eqref{eq:Delta-expanded}}=
\sum_{(T_l)}
T_{l,(1)}\otimes T_{l,(2)}\otimes T_r
+
\sum_{(T_l)}
\Delta_{\rm LR}^{\epsilon}(T_{l,(1)})
\otimes \bigl(T_{l,(2)}\vee T_r\bigr)
\notag
\\
&\quad
+
\sum_{(T_r)}
\Delta_{\rm LR}^{\epsilon}\bigl(T_l\vee T_{r,(1)}\bigr)
\otimes T_{r,(2)}
\notag
\\
&\overset{\eqref{eq:Delta-expanded}}=
\sum_{(T_l)}
T_{l,(1)}\otimes T_{l,(2)}\otimes T_r
+
\sum_{(T_l)}
\Delta_{\rm LR}^{\epsilon}(T_{l,(1)})
\otimes \bigl(T_{l,(2)}\vee T_r\bigr)
\notag
\\
&\quad
+
\sum_{(T_r)}
T_l\otimes T_{r,(1)}\otimes T_{r,(2)}
\notag
\\
&\quad
+
\sum_{(T_l),(T_r)}
T_{l,(1)}\otimes
\bigl(T_{l,(2)}\vee T_{r,(1)}\bigr)
\otimes T_{r,(2)}
\notag
\\
&\quad
+
\sum_{(T_r),(T_{r,(1)})}
\bigl(T_l\vee T_{r,(1)(1)}\bigr)
\otimes T_{r,(1)(2)}
\otimes T_{r,(2)}.
\mlabel{eq:coass-new-2}
\end{align}
Similarly, 
\begin{align} 
(\id\otimes \Delta_{\rm LR}^{\epsilon}) \Delta_{\rm LR}^{\epsilon}(T) &\overset{\eqref{eq:Delta-expanded}}
= \sum_{(T_r)} T_l\otimes T_{r,(1)}\otimes T_{r,(2)} + \sum_{(T_l)} T_{l,(1)}\otimes \Delta_{\rm LR}^{\epsilon} \bigl(T_{l,(2)}\vee T_r\bigr) 
\notag \\ 
&\quad + \sum_{(T_r)} \bigl(T_l\vee T_{r,(1)}\bigr) \otimes \Delta_{\rm LR}^{\epsilon}(T_{r,(2)}) 
\notag \\ &\overset{\eqref{eq:Delta-expanded}}= \sum_{(T_r)} T_l\otimes T_{r,(1)}\otimes T_{r,(2)} + \sum_{(T_l)} T_{l,(1)}\otimes T_{l,(2)}\otimes T_r 
\notag \\ &\quad + \sum_{(T_l),(T_{l,(2)})} T_{l,(1)} \otimes T_{l,(2)(1)} \otimes \bigl(T_{l,(2)(2)}\vee T_r\bigr) \notag \\ &\quad + \sum_{(T_l),(T_r)} T_{l,(1)}\otimes \bigl(T_{l,(2)}\vee T_{r,(1)}\bigr) \otimes T_{r,(2)} \notag \\ &\quad + \sum_{(T_r)} \bigl(T_l\vee T_{r,(1)}\bigr) \otimes \Delta_{\rm LR}^{\epsilon}(T_{r,(2)}). 
\mlabel{eq:coass-new-3} 
\end{align}
We now compare   \meqref{eq:coass-new-2} and
\meqref{eq:coass-new-3}. Write
$T=T_l\vee T_r$, applying the induction hypothesis to
$T_l$ gives
\[
(\Delta_{\rm LR}^{\epsilon}\otimes\id)
\Delta_{\rm LR}^{\epsilon}(T_l)
=
(\id\otimes\Delta_{\rm LR}^{\epsilon})
\Delta_{\rm LR}^{\epsilon}(T_l),
\]
that is,
\[
\sum_{(T_l),(T_{l,(1)})}
T_{l,(1)(1)}\otimes T_{l,(1)(2)}\otimes T_{l,(2)}
=
\sum_{(T_l),(T_{l,(2)})}
T_{l,(1)}\otimes T_{l,(2)(1)}\otimes T_{l,(2)(2)}.
\]
Hence, 
\begin{equation}
\sum_{(T_l)}
\Delta_{\rm LR}^{\epsilon}(T_{l,(1)})
\otimes \bigl(T_{l,(2)}\vee T_r\bigr)
=
\sum_{(T_l),(T_{l,(2)})}
T_{l,(1)}
\otimes T_{l,(2)(1)}
\otimes \bigl(T_{l,(2)(2)}\vee T_r\bigr).
\mlabel{eq:coass-new-4}
\end{equation}
Similarly, 
\begin{equation}
\sum_{(T_r),(T_{r,(1)})}
\bigl(T_l\vee T_{r,(1)(1)}\bigr)
\otimes T_{r,(1)(2)}
\otimes T_{r,(2)}
=
\sum_{(T_r)}
\bigl(T_l\vee T_{r,(1)}\bigr)
\otimes
\Delta_{\rm LR}^{\epsilon}(T_{r,(2)}).
\mlabel{eq:coass-new-5}
\end{equation}
Comparing \meqref{eq:coass-new-2} with \meqref{eq:coass-new-3}, the common
terms are identical, while the remaining terms are identified by
\meqref{eq:coass-new-4} and \meqref{eq:coass-new-5}.  
This proves the induction step and completes the proof.
\end{proof}

We now introduce the under product.  This product will be used to express the
infinitesimal compatibility of $\Delta_{\rm LR}^{\epsilon}$.

\begin{definition}\mcite{Aguiar2006}
The \textbf{under product}\footnotemark{} is the bi-linear map
\[
\backslash:H_{\rm LR}^{\epsilon} \otimes H_{\rm LR}^{\epsilon} \longrightarrow H_{\rm LR}^{\epsilon}, \quad
S\otimes T\longmapsto S\backslash T,
\]
where $S\backslash T$ is obtained by adjoining the root of $T$ to the
rightmost branch of $S$. 

The trivial tree $|$ is the unit, satisfying
$|\backslash T=T=T\backslash |.$
Moreover, if $S\neq |$ and
$S=S_l\vee S_r$
is its root decomposition, then, for every $T\in\Y$, one has
\begin{equation}
S\backslash T=S_l\vee(S_r\backslash T).
\mlabel{eq:backslash-recursion}
\end{equation}
\end{definition}

\footnotetext{We use the term ``under product'' for $\backslash$, following
Aguiar and Sottile~\mcite{Aguiar2006}, rather than the more generic phrase
``associative product''.  This avoids confusing $\backslash$ with other
associative products on planar binary trees, in particular the Loday--Ronco
Hopf product.}

\begin{exam}
Let
$
S=\YY{\yy10}=\YY{}\vee |$ and $T=\YY{}.$ Then,
\[
S\backslash T
=
S_l\vee(S_r\backslash T)
=
\YY{}\vee(|\backslash \YY{})
=
\YY{}\vee \YY{}
=
\YY{\yy20\yy23}.
\]
\end{exam}

A \textbf{right decomposition}~\mcite{Aguiar2006} of a planar binary tree $T$
is a way of writing $T$ as $R\backslash S$.  
The trivial right decompositions
are
$$
T=T\backslash |, \quad
T=|\backslash T.
$$
A non-trivial tree $T$ is called \textbf{progressive}~\mcite{Aguiar2006} if
the trivial right decompositions are its only right decompositions.
Equivalently, if $T=T_l\vee T_r$ is the root decomposition of $T$, then
$T$ is progressive if and only if $T_r=|$.  Thus every non-trivial progressive
tree is of the form $T_l\vee |$.  For instance,
\[
\YY{},
\quad
\YY{\yy10},
\quad
\YY{\yy10\yy21},
\quad
\YY{\yy10\yy21\yy30}.
\]
Moreover, every non-trivial planar binary tree admits a unique factorization
\[
T=U_1\backslash\cdots\backslash U_m,
\quad
U_i\in{\rm P}(\Y),
\quad m\ge 1,
\]
where ${\rm P}(\Y)$ denotes the set of  progressive planar binary
trees.  In particular, Aguiar and Sottile~\mcite{Aguiar2006} showed that
$(H_{\rm LR}^{\epsilon},\backslash,|)$ is the free unitary associative algebra
on ${\rm P}(\Y)$.

We equip $H_{\rm LR}^{\epsilon}\otimes H_{\rm LR}^{\epsilon}$ with the natural
$H_{\rm LR}^{\epsilon}$-bimodule structure induced by the under product
$\backslash$:
\begin{equation} 
\begin{alignedat}{3}
 \lambda:H_{\rm LR}^{\epsilon} \otimes(H_{\rm LR}^{\epsilon}\otimes H_{\rm LR}^{\epsilon}) &\longrightarrow H_{\rm LR}^{\epsilon}\otimes H_{\rm LR}^{\epsilon}, &\quad U\otimes(V\otimes W) &\longmapsto (U\backslash V)\otimes W,
  \\ \rho:(H_{\rm LR}^{\epsilon} \otimes H_{\rm LR}^{\epsilon})\otimes H_{\rm LR}^{\epsilon} &\longrightarrow H_{\rm LR}^{\epsilon} \otimes H_{\rm LR}^{\epsilon}, &\quad (U\otimes V)\otimes W &\longmapsto U\otimes(V\backslash W). 
  \end{alignedat} 
 \mlabel{de:bimo} \end{equation}

\begin{lemma}
For all $S,T\in H_{\rm LR}^{\epsilon}$, one has
\begin{equation}
\Delta_{\rm LR}^{\epsilon}(S\backslash T)
=
S\backslash \Delta_{\rm LR}^{\epsilon}(T)
+
\Delta_{\rm LR}^{\epsilon}(S)\backslash T.
\mlabel{eq:derivation}
\end{equation}
\mlabel{lem:derivation}
\end{lemma}

\begin{proof}
It is enough to prove the identity for
$S,T\in\Y$.  We prove it by induction on $|{\rm Int}(S)|\geq 0$.
For the initial step of $|{\rm Int}(S)|=0$, we have $S=|$.  Hence
\[
\Delta_{\rm LR}^{\epsilon}(|\backslash T)
=
\Delta_{\rm LR}^{\epsilon}(T)
=
|\backslash \Delta_{\rm LR}^{\epsilon}(T)
+
\Delta_{\rm LR}^{\epsilon}(|)\backslash T.
\]
For the induction step of $|{\rm Int}(S)|\ge 1$, on the one hand,
\begin{align}
\Delta_{\rm LR}^{\epsilon}(S\backslash T)
&\overset{\eqref{eq:backslash-recursion}}=
\Delta_{\rm LR}^{\epsilon}\bigl(S_l\vee(S_r\backslash T)\bigr)
\notag
\\
&\overset{\eqref{eq:yucheng}}=
S_l\otimes(S_r\backslash T)
+
\sum_{(S_l)}
S_{l,(1)}\otimes
\bigl(S_{l,(2)}\vee(S_r\backslash T)\bigr)
\notag
\\
&\quad
+
(\vee\otimes\id)
\bigl(S_l\otimes
\Delta_{\rm LR}^{\epsilon}(S_r\backslash T)\bigr)
\notag
\\
&=
S_l\otimes(S_r\backslash T)
+
\sum_{(S_l)}
S_{l,(1)}\otimes
\bigl(S_{l,(2)}\vee(S_r\backslash T)\bigr)
\notag
\\
&\quad
+
(\vee\otimes\id)
\bigl(S_l\otimes
(S_r\backslash\Delta_{\rm LR}^{\epsilon}(T)
+
\Delta_{\rm LR}^{\epsilon}(S_r)\backslash T)\bigr)
\hspace{1cm} \text{(by the inductive hypothesis)}
\notag
\\
&=
S_l\otimes(S_r\backslash T)
+
\sum_{(S_l)}
S_{l,(1)}\otimes
\bigl(S_{l,(2)}\vee(S_r\backslash T)\bigr)
\notag
\\
&\quad
+
(\vee\otimes\id)
\bigl(S_l\otimes
(S_r\backslash\Delta_{\rm LR}^{\epsilon}(T))\bigr)
+
(\vee\otimes\id)
\bigl(S_l\otimes
(\Delta_{\rm LR}^{\epsilon}(S_r)\backslash T)\bigr).
\mlabel{eq:zhankai}
\end{align}
On the other hand, 
\begin{align}
S\backslash\Delta_{\rm LR}^{\epsilon}(T)
&=
\sum_{(T)}
(S\backslash T_{(1)})\otimes T_{(2)}
\notag
\\
&\overset{\eqref{eq:backslash-recursion}}=
\sum_{(T)}
\bigl(S_l\vee(S_r\backslash T_{(1)})\bigr)\otimes T_{(2)}
\notag
\\
&=
(\vee\otimes\id)
\bigl(S_l\otimes
(S_r\backslash\Delta_{\rm LR}^{\epsilon}(T))\bigr),
\mlabel{eq:dao}
\end{align}
and
\begin{align}
\Delta_{\rm LR}^{\epsilon}(S)\backslash T
&\overset{\eqref{eq:yucheng}}=
\begin{aligned}[t]
\Bigl(
&S_l\otimes S_r
+
\sum_{(S_l)}
S_{l,(1)}\otimes (S_{l,(2)}\vee S_r)
\\
&+
\sum_{(S_r)}
(S_l\vee S_{r,(1)})\otimes S_{r,(2)}
\Bigr)\backslash T
\end{aligned}
\notag
\\
&\overset{\eqref{de:bimo}}=
\begin{aligned}[t]
&S_l\otimes(S_r\backslash T)
+
\sum_{(S_l)}
S_{l,(1)}\otimes
\bigl((S_{l,(2)}\vee S_r)\backslash T\bigr)
\\
&+
\sum_{(S_r)}
(S_l\vee S_{r,(1)})\otimes(S_{r,(2)}\backslash T)
\end{aligned}
\notag
\\
&\overset{\eqref{eq:backslash-recursion}}=
\begin{aligned}[t]
&S_l\otimes(S_r\backslash T)
+
\sum_{(S_l)}
S_{l,(1)}\otimes
\bigl(S_{l,(2)}\vee(S_r\backslash T)\bigr)
\\
&+
\sum_{(S_r)}
(S_l\vee S_{r,(1)})\otimes(S_{r,(2)}\backslash T)
\end{aligned}
\notag
\\
&=
\begin{aligned}[t]
&S_l\otimes(S_r\backslash T)
+
\sum_{(S_l)}
S_{l,(1)}\otimes
\bigl(S_{l,(2)}\vee(S_r\backslash T)\bigr)
\\
&+
(\vee\otimes\id)
\bigl(S_l\otimes
(\Delta_{\rm LR}^{\epsilon}(S_r)\backslash T)\bigr).
\end{aligned}
\mlabel{eq:deltst}
\end{align}
Substituting \meqref{eq:dao} and \meqref{eq:deltst} into
\meqref{eq:zhankai}, we complete the proof.
\end{proof}

\subsection{An elementary vertex-cut formula for the infinitesimal coproduct}
\mlabel{subsec:combinatorial}

We now give a combinatorial interpretation of the coproduct
$\Delta_{\rm LR}^{\epsilon}$.  

\begin{definition}
Let $T\in\Y$ and $e\in{\rm Int}(T)$.  The {\bf elementary cut} of $T$ at $e$
is specified by a pair
\[
\bigl(P^e(T),R^e(T)\bigr)\in\Y\times\Y,
\]
defined recursively on $|{\rm Int}(T)|$ according to the position of $e$ in
the root decomposition of $T$.

There is no elementary cut on the trivial tree $|$.  The initial case is
$|{\rm Int}(T)|=1$.  Then $T=|\vee |$ and
${\rm Int}(T)=\{\rho_T\}$.  We set
\[
\bigl(P^{\rho_T}(T),R^{\rho_T}(T)\bigr):=(|,|).
\]
Now assume that $|{\rm Int}(T)|\ge 2$ and write
$T=T_l\vee T_r$ with
${\rm Int}(T)=\{\rho_T\}\sqcup{\rm Int}(T_l)\sqcup{\rm Int}(T_r).$
We define
\begin{equation}
\bigl(P^e(T),R^e(T)\bigr):=
\left\{
\begin{aligned}
&(T_l,T_r),
&&  \quad \text{if} \, e=\rho_T,\\
&\bigl(P^e(T_l),\,R^e(T_l)\vee T_r\bigr),
&& \quad \text{if} \, e\in{\rm Int}(T_l),\\
&\bigl(T_l\vee P^e(T_r),\,R^e(T_r)\bigr),
&& \quad \text{if} \, e\in{\rm Int}(T_r).
\end{aligned}
\right.
\mlabel{eq:elementary-vertex-cut-recursion}
\end{equation}
\mlabel{def:elementary-vertex-cut}
\end{definition}

\begin{remark}
The elementary cut in Definition~\ref{def:elementary-vertex-cut} is not an admissible cut in the sense of the
Connes--Kreimer~\mcite{Connes1998,Kreimer1998} or Loday--Ronco Hopf coproducts~\mcite{Loday1998, Zhang2020}.  
\end{remark}

The following formula shows that these elementary cuts give exactly the
coproduct $\Delta_{\rm LR}^{\epsilon}$.

\begin{theorem}
For every planar binary tree $T\in\Y$, one has
\begin{equation}
\Delta_{\rm LR}^{\epsilon}(T)
=
\sum_{e\in {\rm Int}(T)} P^e(T)\otimes R^e(T).
\mlabel{eq:cut-formula}
\end{equation}
\mlabel{thm:cut-formula}
\end{theorem}

\begin{proof}
We prove \meqref{eq:cut-formula} by induction on $|{\rm Int}(T)|\ge 0$. For the initial step of $|{\rm Int}(T)|=0$, one has $T=|$ and
${\rm Int}(T)=\emptyset$.  Hence
\[
\Delta_{\rm LR}^{\epsilon}(T)
=
\Delta_{\rm LR}^{\epsilon}(|)
=
0
=
\sum_{e\in{\rm Int}(T)}P^e(T)\otimes R^e(T).
\]
For the induction step of $|{\rm Int}(T)|\ge 1$, 
write $T=T_l\vee T_r.$ Then
\begin{align*}
\Delta_{\rm LR}^{\epsilon}(T)
&\overset{\eqref{eq:yucheng}}=
T_l\otimes T_r
+
(\id\otimes \vee)
\bigl(\Delta_{\rm LR}^{\epsilon}(T_l)\otimes T_r\bigr)
\\
&\quad
+
(\vee\otimes \id)
\bigl(T_l\otimes \Delta_{\rm LR}^{\epsilon}(T_r)\bigr)
\\
&=
T_l\otimes T_r
+
\sum_{u\in {\rm Int}(T_l)}
P^{u}(T_l)\otimes \bigl(R^{u}(T_l)\vee T_r\bigr)
\\
&\quad
+
\sum_{w\in {\rm Int}(T_r)}
\bigl(T_l\vee P^{w}(T_r)\bigr)\otimes R^{w}(T_r)
\hspace{1cm}\text{(by the inductive hypothesis)}
\\
&=
P^{\rho_T}(T)\otimes R^{\rho_T}(T)
+
\sum_{u\in {\rm Int}(T_l)}
P^{u}(T)\otimes R^{u}(T)
\\
&\quad
+
\sum_{w\in {\rm Int}(T_r)}
P^{w}(T)\otimes R^{w}(T)
\hspace{1cm}\text{(by Definition~\mref{def:elementary-vertex-cut})}
\\
&=
\sum_{e\in {\rm Int}(T)}
P^{e}(T)\otimes R^{e}(T)
\hspace{1cm}\text{(by the decomposition of ${\rm Int}(T)$)}.
\end{align*}
This proves the induction step and completes the proof.
\end{proof}

To make the combinatorial content of the coproduct formula transparent, we
record the following examples.
\begin{exam}
\noindent
Let
\[
T_1=
\YY{\yy10\yy20
\node[scale=0.8] at (0.75,-0.75) {${\rho}$};
\node[scale=0.8] at (0.25,-0.45) {${v}$};
\node[scale=0.8] at (-0.20,-0.15) {${w}$};
}
\]
with ${\rm Int}(T_1)=\{\rho,v,w\}$. Then
\[
\begin{array}{c|c|c}
e & P^{e}(T_1) & R^{e}(T_1)
\\ \hline
v & \YY{} & \YY{}
\\[1mm]
w & | & \YY{\yy10}
\\[1mm]
\rho & \YY{\yy10} & |
\end{array}
\]
Hence
\begin{align*}
\Delta_{\rm LR}^{\epsilon}(T_1)
&=
P^{v}(T_1)\otimes R^{v}(T_1)
+
P^{w}(T_1)\otimes R^{w}(T_1)
+
P^{\rho}(T_1)\otimes R^{\rho}(T_1)
\\
&=
\YY{}\otimes \YY{}
+
|\otimes \YY{\yy10}
+
\YY{\yy10}\otimes |.
\end{align*}
This agrees with Example~\mref{exam:coproduct-basic}.
\mlabel{ex:cut-formula}
\end{exam}

We now state the resulting infinitesimal bialgebra structure.
\begin{theorem}
The quadruple $(H_{\rm LR}^{\epsilon},\backslash,|,\Delta_{\rm LR}^{\epsilon})$
is a unitary infinitesimal bialgebra.
\mlabel{thm:inf-bialgebra}
\end{theorem}
\begin{proof}
Notice that
$(H_{\rm LR}^{\epsilon},\backslash,|)$ is a unitary associative algebra~\mcite{Aguiar2006}.
Then the result follows from  Proposition~\mref{lem:coassociative} and Lemma~\mref{lem:derivation}.
\end{proof}

\begin{remark}
The coproduct $\Delta_{\rm LR}^{\epsilon}$ is not counitary.  Indeed, since
$\Delta_{\rm LR}^{\epsilon}(|)=0,$
there cannot exist a linear map $\varepsilon:H_{\rm LR}^{\epsilon}\to\bk$ satisfying
\[
(\varepsilon\otimes\id)\Delta_{\rm LR}^{\epsilon}=\id_{H_{\rm LR}^{\epsilon}}
\quad\text{and}\quad
(\id\otimes\varepsilon)\Delta_{\rm LR}^{\epsilon}=\id_{H_{\rm LR}^{\epsilon}}.
\]
Applying the first identity to $|$ would give $0=|$, a contradiction.
\end{remark}

\section{Free unitary infinitesimal \texorpdfstring{$(\backslash,\vee)$}{(backslash,vee)}-bialgebras}
\mlabel{sec:free}

In this section we formulate a universal property   for the unitary infinitesimal
bialgebra obtained in Theorem \ref{thm:inf-bialgebra}.

\subsection{Two one-sided grafting operators}
\mlabel{subsec:onesided}
The binary grafting operation $\vee$ gives two natural one-sided operators by
fixing the trivial tree on the left or on the right.

\begin{definition}
Define two linear maps
\[
B^{r},B^{\ell}:H_{\rm LR}^{\epsilon}
\longrightarrow
H_{\rm LR}^{\epsilon}
\]
by setting, for every tree $T\in\Y$,
\[
B^{r}(T):=T\vee |,
\quad
B^{\ell}(T):=|\vee T.
\]
\mlabel{def:BellBr}
\end{definition}
\begin{exam}
For the one-vertex tree, one has
\[
B^{r}(\YY{})
=
\YY{}\vee |
=
\YY{\yy10},
\quad
B^{\ell}(\YY{})
=
|\vee \YY{}
=
\YY{\yy11}.
\]
Iterating once more gives the two comb examples
\[
B^{r}(\YY{\yy10})
=
\YY{\yy10}\vee |
=
\YY{\yy10\yy20},
\quad
B^{\ell}(\YY{\yy11})
=
|\vee \YY{\yy11}
=
\YY{\yy11\yy23}.
\]
\mlabel{exam:BellBr}
\end{exam}
The coproduct $\Delta_{\rm LR}^{\epsilon}$ behaves recursively with respect to these two one-sided
grafting operators.
\begin{prop}
For every $T\in H_{\rm LR}^{\epsilon}$, one has
\begin{equation}
\Delta_{\rm LR}^{\epsilon}(B^{r}(T))
=
T\otimes |
+
(\id\otimes B^{r})\Delta_{\rm LR}^{\epsilon}(T),
\mlabel{eq:Bell-cocycle-type}
\end{equation}
and
\begin{equation}
\Delta_{\rm LR}^{\epsilon}(B^{\ell}(T))
=
|\otimes T
+
(B^{\ell}\otimes \id)\Delta_{\rm LR}^{\epsilon}(T).
\mlabel{eq:Br-cocycle-type}
\end{equation}
\mlabel{prop:BellBr-coproduct}
\end{prop}

\begin{proof}
It is enough to verify the identities on $T\in\Y$.  For
\meqref{eq:Bell-cocycle-type}, using Definition~\mref{def:BellBr} and
\meqref{eq:yucheng}, 
\begin{align*}
\Delta_{\rm LR}^{\epsilon}(B^{r}(T))
&=
\Delta_{\rm LR}^{\epsilon}(T\vee |)
\\
&\overset{\eqref{eq:yucheng}}=
T\otimes |
+(\id\otimes \vee)\bigl(\Delta_{\rm LR}^{\epsilon}(T)\otimes |\bigr)
+(\vee\otimes \id)\bigl(T\otimes \Delta_{\rm LR}^{\epsilon}(|)\bigr)
\\
&=
T\otimes |
+(\id\otimes \vee)\bigl(\Delta_{\rm LR}^{\epsilon}(T)\otimes |\bigr)
\\
&=
T\otimes |
+
(\id\otimes B^{r})\Delta_{\rm LR}^{\epsilon}(T).
\end{align*}
The proof of \meqref{eq:Br-cocycle-type} is similar.
This completes the proof.
\end{proof}

For the comb trees introduced in~\eqref{eq:right-comb} and~\eqref{eq:left-comb}, the
one-sided grafting operators give
\[
F_n^{r}=(B^{r})^n(|),
\quad
F_n^{\ell}=(B^{\ell})^n(|),
\quad \forall \,  n\ge 0.
\]
As a consequence of Proposition~\mref{prop:BellBr-coproduct}, we obtain
the coproduct formulas for these comb trees.
\begin{coro}
For every $n\ge 0$, one has
\begin{align}
\Delta_{\rm LR}^{\epsilon}(F_n^{r})
&=
\sum_{i=0}^{n-1} F_{n-1-i}^{r}\otimes F_i^{r},
\mlabel{eq:ell-comb-coproduct}
\\
\Delta_{\rm LR}^{\epsilon}(F_n^{\ell})
&=
\sum_{i=0}^{n-1} F_i^{\ell}\otimes F_{n-1-i}^{\ell},
\mlabel{eq:r-comb-coproduct}
\end{align}
where the sum is understood to be zero for $n=0$.
\mlabel{cor:comb-coproduct}
\end{coro}

\begin{proof}
We prove \meqref{eq:ell-comb-coproduct} by induction on $n \geq 0$.
For the initial step of $n=0$, we have $F_0^{r}=|$, and hence
\[
\Delta_{\rm LR}^{\epsilon}(F_0^{r})
=
\Delta_{\rm LR}^{\epsilon}(|)
=
0.
\]
For the induction step of $n\ge 1$, assume that
\[
\Delta_{\rm LR}^{\epsilon}(F_{n-1}^{r})
=
\sum_{i=0}^{n-2}F_{n-2-i}^{r}\otimes F_i^{r}.
\]
Since $F_n^{r}=B^{r}(F_{n-1}^{r})$, Proposition~\mref{prop:BellBr-coproduct} gives
\begin{align*}
\Delta_{\rm LR}^{\epsilon}(F_n^{r})
&\overset{\eqref{eq:Bell-cocycle-type}}=
F_{n-1}^{r}\otimes |
+
(\id\otimes B^{r})
\Delta_{\rm LR}^{\epsilon}(F_{n-1}^{r})
\\
&=
F_{n-1}^{r}\otimes F_0^{r}
+
\sum_{i=0}^{n-2}
F_{n-2-i}^{r}\otimes B^{r}(F_i^{r})
\hspace{1cm}\text{(by the inductive hypothesis)}
\\
&=
F_{n-1}^{r}\otimes F_0^{r}
+
\sum_{i=0}^{n-2}
F_{n-2-i}^{r}\otimes F_{i+1}^{r}
\\
&=
\sum_{i=0}^{n-1}F_{n-1-i}^{r}\otimes F_i^{r}.
\end{align*}
Thus \meqref{eq:ell-comb-coproduct} holds.  The proof of
\meqref{eq:r-comb-coproduct} is similar, using
\meqref{eq:Br-cocycle-type}.
\end{proof}

\subsection{Unitary \texorpdfstring{$(\backslash,\vee)$}{(backslash,vee)}-algebras of planar binary trees}
\mlabel{subsec:vee-alg}

We now isolate the algebraic structure carried by planar binary trees with
the operations $\backslash$ and $\vee$.  This allows us to formulate the
universal property satisfied by
$(H_{\rm LR}^{\epsilon},\backslash,|,\vee).$

The following definition is inspired by the operated-algebra and cocycle
bialgebra viewpoints used for rooted forests \mcite{Guo2009,zhang2019,Zhang2022a} and
by the compatibility between the under product $\backslash$ and the
binary grafting operation $\vee$ on planar binary trees
\mcite{Aguiar2006}.

\begin{definition}
A {\bf unitary $(\backslash,\vee)$-algebra} is a unitary associative algebra
$(H,\backslash_{H},1_H)$
together with a bi-linear map
\[
\vee_H:H\otimes H\to H
\]
such that
\begin{equation}
(x\vee_H y)\backslash_{H} z
=
x\vee_H(y\backslash_{H} z),
\quad \forall\, x,y,z\in H.
\mlabel{eq:vee-backslash-compat}
\end{equation}
A morphism of unitary $(\backslash,\vee)$-algebras is defined as usual.
\mlabel{def:vee-algebra}
\end{definition}

\begin{lemma}
The quadruple
$(H_{\rm LR}^{\epsilon},\backslash,|,\vee)$
is a unitary $(\backslash,\vee)$-algebra.
\mlabel{lem:A-unitary-vee}
\end{lemma}

\begin{proof}
The triple $(H_{\rm LR}^{\epsilon},\backslash,|)$ is a unitary associative algebra by Aguiar--Sottile~\mcite{Aguiar2006}, and the compatibility~\meqref{eq:vee-backslash-compat} follows from~\meqref{eq:backslash-recursion}.
\end{proof}

Now we arrive at our main conclusion of this subsection.
\begin{theorem}
The unitary $(\backslash,\vee)$-algebra
$(H_{\rm LR}^{\epsilon},\backslash,|,\vee)$
is  the free unitary
$(\backslash,\vee)$-algebra on the empty set.  
\mlabel{thm:free-vee-empty}
\end{theorem}

\begin{proof}
({\bf Existence}). 
Let $(H,\backslash_H,1_H,\vee_H)$ be a unitary
$(\backslash,\vee)$-algebra.  Define the linear map
\[
\Phi:H_{\rm LR}^{\epsilon}\longrightarrow H
\]
on basis trees recursively  by
\begin{equation}
\Phi(|):=1_H,
\quad
\Phi(T_l\vee T_r)
:=
\Phi(T_l)\vee_H\Phi(T_r),
\quad \forall \, T_l,T_r\in\Y.
\mlabel{eq:Phi-recursion}
\end{equation}
We now  verify that $\Phi$ preserves
$\backslash$
\begin{equation}
\Phi(S\backslash T)
=
\Phi(S)\backslash_H\Phi(T),
\quad \forall \, S,T\in\Y,
\mlabel{eq:Phi-backslash}
\end{equation}
by induction on $|{\rm Int}(S)|\geq 0$.
For the initial step of $|{\rm Int}(S)|=0$, one has $S=|$, and hence
\[
\Phi(|\backslash T)
=
\Phi(T)
=
1_H\backslash_H\Phi(T)
=
\Phi(|)\backslash_H\Phi(T).
\]
For the induction step of $|{\rm Int}(S)|\ge 1$, write $S=S_l\vee S_r$.  Then
\begin{align*}
\Phi(S\backslash T)
&\overset{\eqref{eq:backslash-recursion}}=
\Phi\bigl(S_l\vee(S_r\backslash T)\bigr)
\\
&\overset{\eqref{eq:Phi-recursion}}=
\Phi(S_l)\vee_H\Phi(S_r\backslash T)
\\
&=
\Phi(S_l)\vee_H
\bigl(\Phi(S_r)\backslash_H\Phi(T)\bigr)
\hspace{1cm}\text{(by the inductive hypothesis)}
\\
&\overset{\eqref{eq:vee-backslash-compat}}=
\bigl(\Phi(S_l)\vee_H\Phi(S_r)\bigr)
\backslash_H\Phi(T)
\\
&\overset{\eqref{eq:Phi-recursion}}=
\Phi(S_l\vee S_r)\backslash_H\Phi(T)
\\
&=
\Phi(S)\backslash_H\Phi(T).
\end{align*}
Thus $\Phi$ is compatible with the products $\backslash$ and $\backslash_H$.
By \meqref{eq:Phi-recursion}, it also preserves the unit and the grafting
operation.  Hence $\Phi$ is a morphism of unitary
$(\backslash,\vee)$-algebras.

({\bf Uniqueness}).  It follows from \meqref{eq:Phi-recursion}.  
\end{proof}

\subsection{Unitary infinitesimal \texorpdfstring{$(\backslash,\vee)$}{(backslash,vee)}-bialgebras}
\mlabel{subsec:inf-vee-bialg}

We now add the infinitesimal coproduct to the preceding
$(\backslash,\vee)$-algebraic structure and require it to be compatible with
the grafting operation.

\begin{definition}
A {\bf unitary infinitesimal $(\backslash,\vee)$-bialgebra} is a quintuple
$(H,\backslash_{H},1_H,\Delta_H,\vee_H)$
such that

(a) $(H,\backslash_{H},1_H,\vee_H)$ is a unitary
$(\backslash,\vee)$-algebra;

(b) $(H,\backslash_{H},1_H,\Delta_H)$ is a unitary infinitesimal bialgebra; 

(c) The operation $\vee_{H}$ satisfies the compatibility
\begin{equation}
\Delta_H(x\vee_H y)
=
x\otimes y
+
(\id\otimes \vee_H)\bigl(\Delta_H(x)\otimes y\bigr)
+
(\vee_H\otimes \id)\bigl(x\otimes \Delta_H(y)\bigr),
\quad \forall\, x,y\in H.
\mlabel{eq:vee-recursive-compat}
\end{equation}

A morphism of unitary infinitesimal $(\backslash,\vee)$-bialgebras is defined as usual.
\mlabel{def:inf-vee-bialgebra}
\end{definition}

\begin{prop}
The quintuple
$(H_{\rm LR}^{\epsilon},\backslash,|,\Delta_{\rm LR}^{\epsilon},\vee)$
is a unitary infinitesimal $(\backslash,\vee)$-bialgebra.
\mlabel{prop:HLR-inf-vee}
\end{prop}
\begin{proof}
By Lemma~\mref{lem:A-unitary-vee},
$(H_{\rm LR}^{\epsilon},\backslash,|,\vee)
$
is a unitary $(\backslash,\vee)$-algebra.  By
Theorem~\mref{thm:inf-bialgebra},
$(H_{\rm LR}^{\epsilon},\backslash,|,\Delta_{\rm LR}^{\epsilon})$
is a unitary infinitesimal bialgebra. The compatibility \meqref{eq:vee-recursive-compat} follows directly from
\meqref{eq:yucheng}.
\end{proof}

Having identified the planar binary-tree object inside this category, we now
record its corresponding universal property.

\begin{theorem}
The unitary infinitesimal $(\backslash,\vee)$-bialgebra
$(H_{\rm LR}^{\epsilon},\backslash,|,
\Delta_{\rm LR}^{\epsilon},\vee)$
is the free unitary infinitesimal $(\backslash,\vee)$-bialgebra on the empty
set.
\mlabel{thm:free-inf-vee}
\end{theorem}

\begin{proof}
Let
$(H,\backslash_{H},1_H,\Delta_H,\vee_H)$
be a unitary infinitesimal $(\backslash,\vee)$-bialgebra.  By
Theorem~\mref{thm:free-vee-empty}, there exists a unique morphism $\Phi:H_{\rm LR}^{\epsilon}\longrightarrow H$ of unitary
$(\backslash,\vee)$-algebras
satisfying
\[
\Phi(|)=1_H,
\quad
\Phi(T_l\vee T_r)=\Phi(T_l)\vee_H\Phi(T_r),
\quad \forall\, T_l,T_r\in\Y.
\]
It remains to prove that $\Phi$ is compatible with the coproducts, namely,
\begin{equation}
(\Phi\otimes \Phi)\Delta_{\rm LR}^{\epsilon}(T)
=
\Delta_H(\Phi(T)),
\quad \forall\, T\in \Y.
\mlabel{eq:Phi-compatible-coproduct}
\end{equation}
We prove \meqref{eq:Phi-compatible-coproduct} by induction on
$|{\rm Int}(T)|\geq 0$.
For the initial step of $\lvert{\rm Int}(T)\rvert=0$, we have $T=|$.  Since
$\Delta_{\rm LR}^{\epsilon}(|)=0,$
we get
\[
(\Phi\otimes \Phi)\Delta_{\rm LR}^{\epsilon}(|)=0.
\]
Also, the infinitesimal compatibility in $H$ gives
\[
\Delta_H(1_H)
=
\Delta_H(1_H\backslash_H 1_H)
\overset{\eqref{eq:infinitesimal-compatibility}}=
1_H\backslash_H \Delta_H(1_H)
+
\Delta_H(1_H)\backslash_H 1_H
=
2\Delta_H(1_H),
\]
hence $\Delta_H(1_H)=0$.  Since $\Phi(|)=1_H$, we get
\[
\Delta_H(\Phi(|))=0.
\]
Thus \meqref{eq:Phi-compatible-coproduct} holds for
$\lvert{\rm Int}(T)\rvert=0$.

For the induction step of $|{\rm Int}(T)|\ge 1$, write
$T=T_l\vee T_r$.  We compute
\begin{align*}
(\Phi\otimes \Phi)\Delta_{\rm LR}^{\epsilon}(T)
&\overset{\eqref{eq:yucheng}}=
(\Phi\otimes \Phi)\Bigl(
T_l\otimes T_r
+
(\id\otimes \vee)
\bigl(\Delta_{\rm LR}^{\epsilon}(T_l)\otimes T_r\bigr)
\\
&\quad
+
(\vee\otimes \id)
\bigl(T_l\otimes \Delta_{\rm LR}^{\epsilon}(T_r)\bigr)
\Bigr)
\\
&\overset{\eqref{eq:Phi-recursion}}=
\Phi(T_l)\otimes \Phi(T_r)
+
(\id\otimes \vee_H)
\bigl((\Phi\otimes \Phi)\Delta_{\rm LR}^{\epsilon}(T_l)
\otimes \Phi(T_r)\bigr)
\\
&\quad
+
(\vee_H\otimes \id)
\bigl(\Phi(T_l)\otimes
(\Phi\otimes \Phi)\Delta_{\rm LR}^{\epsilon}(T_r)\bigr)
\\
&=
\Phi(T_l)\otimes \Phi(T_r)
+
(\id\otimes \vee_H)
\bigl(\Delta_H(\Phi(T_l))\otimes \Phi(T_r)\bigr)
\\
&\quad
+
(\vee_H\otimes \id)
\bigl(\Phi(T_l)\otimes \Delta_H(\Phi(T_r))\bigr)
\hspace{1cm}\text{(by the inductive hypothesis)}
\\
&\overset{\eqref{eq:vee-recursive-compat}}=
\Delta_H\bigl(\Phi(T_l)\vee_H \Phi(T_r)\bigr)
\\
&\overset{\eqref{eq:Phi-recursion}}=
\Delta_H\bigl(\Phi(T_l\vee T_r)\bigr)
\\
&=
\Delta_H(\Phi(T)).
\end{align*}
This proves \meqref{eq:Phi-compatible-coproduct} for
$\lvert{\rm Int}(T)\rvert\ge 1$ and completes the induction.

Therefore $\Phi$ preserves the coproduct.  Since $\Phi$ is already a morphism
of unitary $(\backslash,\vee)$-algebras, it is a morphism of unitary
infinitesimal $(\backslash,\vee)$-bialgebras.  Finally, uniqueness follows
from Theorem~\mref{thm:free-vee-empty}.  This completes the proof.
\end{proof}

\section{A bridge from planar rooted forests to planar binary trees}\mlabel{sec:bridge}

In this section we compare the rooted-forest infinitesimal bialgebra $(\bk\mathcal F,\cdot,\funit,\Delta_{\rm RT}^{\epsilon})$~\mcite{Zhang2022a}
 with the
planar binary-tree infinitesimal bialgebra $(H_{\rm LR}^{\epsilon},\backslash,|,\Delta_{\rm LR}^{\epsilon})$ constructed above.  We build an
algebra isomorphism
$\Theta:\bk\mathcal F\longrightarrow H_{\rm LR}^{\epsilon}$
using the recursive structure of rooted forests and the
$\backslash$-factorization of planar binary trees, and then prove that
$\Theta$ transports $\Delta_{\rm RT}^{\epsilon}$ to
$\Delta_{\rm LR}^{\epsilon}$.

Let $\mathcal T$ and $\mathcal F$ denote the sets of planar rooted trees and
planar rooted forests, respectively.    The grafting operation
\[
B^+:\mathcal F\longrightarrow \mathcal T\subseteq \mathcal F
\]
sends a forest to the tree obtained by grafting its roots to a new common
root.   

We recall from~\mcite{Zhang2022a} that the $\bk$-module $\bk\mathcal F$,
equipped with forest concatenation product $\cdot$, unit $\funit$, and the
infinitesimal coproduct
\[
\Delta_{\rm RT}^{\epsilon}:\bk\mathcal F\longrightarrow
\bk\mathcal F\otimes \bk\mathcal F,
\]
is a unitary infinitesimal bialgebra (of weight zero).  Moreover,
\begin{equation}
\Delta_{\rm RT}^{\epsilon}(FG)
=
F\cdot \Delta_{\rm RT}^{\epsilon}(G)
+
\Delta_{\rm RT}^{\epsilon}(F)\cdot G,
\quad \forall \, F,G\in \bk\mathcal F,
\mlabel{eq:rt-derivation}
\end{equation}
and
\begin{equation}
\Delta_{\rm RT}^{\epsilon}(B^+(F))
=
F\otimes \funit
+
(\id\otimes B^+)\Delta_{\rm RT}^{\epsilon}(F),
\quad \forall \, F\in \bk\mathcal F.
\mlabel{eq:rt-cocycle}
\end{equation}
In particular,
$\Delta_{\rm RT}^{\epsilon}(\funit)=0$~\mcite{Zhang2022a}.
Under the correspondence used below, the grafting operator $B^+$ corresponds
to the operator $B^{r}$
on planar binary trees, and forest concatenation corresponds to the
associative product $\backslash$ on $H_{\rm LR}^{\epsilon}$.  Hence the
infinitesimal coproduct $\Delta_{\rm RT}^{\epsilon}$ should correspond, via a
bridge map, to $\Delta_{\rm LR}^{\epsilon}$.

The key point is that on the rooted-forest side the coproduct is governed by
the infinitesimal $1$-cocycle condition for $B^+$ together with the derivation
law for concatenation, whereas on the planar binary tree side
$\Delta_{\rm LR}^{\epsilon}$ is governed by the corresponding recursive identity
for $B^{r}$ together with the derivation law for $\backslash$. 

\begin{remark}
There are two reasons for using $B^r$, rather than $B^\ell$, in the forthcoming proposition.
\begin{enumerate}
\item The $B^r$-recursion~(\mref{eq:uniqblone}) below 
has the same one-sided form as the infinitesimal $1$-cocycle formula for
planar rooted forests~\mcite{Zhang2022a}:
\[
\Delta_{\rm RT}^{\epsilon}(B^+(F))
=
F\otimes \funit
+
(\id\otimes B^+)\Delta_{\rm RT}^{\epsilon}(F).
\]
By contrast, the $B^\ell$-identity \meqref{eq:Br-cocycle-type} has the
opposite form
\[
\Delta_{\rm LR}^{\epsilon}\bigl(B^\ell(T)\bigr)
=
|\otimes T
+
(B^\ell\otimes\id)\Delta_{\rm LR}^{\epsilon}(T),
\]
and therefore does not match the rooted-forest $1$-cocycle pattern.

\item The choice of $B^r$ is made to be compatible with the progressive tree:
\[
{\rm P}(\Y)
=
\{ B^r(U) \mid U\in\Y\}.
\]
Here $B^r$ parametrizes the generators on which the uniqueness argument starts,
whereas $B^\ell$ does not.
\end{enumerate}
\end{remark}

This leads to the following uniqueness characterization of
$\Delta_{\rm LR}^{\epsilon}$.

\begin{prop}
The coproduct $\Delta_{\rm LR}^{\epsilon}$ is uniquely determined by the
following three conditions:
\begin{align}
\Delta_{\rm LR}^{\epsilon}(|)
&=0,
\notag
\\
\Delta_{\rm LR}^{\epsilon}\bigl(B^{r}(T)\bigr)
&=
T\otimes |
+
(\id\otimes B^{r})\Delta_{\rm LR}^{\epsilon}(T),
\quad \forall \, T\in H_{\rm LR}^{\epsilon},
\mlabel{eq:uniqblone}
\\
\Delta_{\rm LR}^{\epsilon}(S\backslash T)
&=
S\backslash \Delta_{\rm LR}^{\epsilon}(T)
+
\Delta_{\rm LR}^{\epsilon}(S)\backslash T,
\quad \forall \, S\in H_{\rm LR}^{\epsilon},\ T\in \{|\}\sqcup {\rm P}(\Y).
\mlabel{eq:uniqbltwo}
\end{align}
\mlabel{prop:uniqbl}
\end{prop}

\begin{proof}
Let
\[
\widetilde{\Delta}:H_{\rm LR}^{\epsilon}\to
H_{\rm LR}^{\epsilon}\otimes H_{\rm LR}^{\epsilon}
\]
be a coproduct satisfying the same three conditions.  We want to prove that $\Delta_{\rm LR}^{\epsilon} = \widetilde{\Delta}$.
Notice that $(H_{\rm LR}^{\epsilon},\backslash,|)$ is freely generated by progressive
trees as a unitary associative algebra~\mcite{Aguiar2006}. So we are left to verify that 
\[
\widetilde{\Delta}(T)
=
\Delta_{\rm LR}^{\epsilon}(T),
\quad
\forall\, T\in \{|\,\}\sqcup{\rm P}(\Y),
\]
which is done by induction on $|{\rm Int}(T)|\ge 0$.

For the initial step of $|{\rm Int}(T)|=0$, one has $T=|$.  Hence
\[
\widetilde{\Delta}(T)
=
\widetilde{\Delta}(|)
=
0
=
\Delta_{\rm LR}^{\epsilon}(|)
=
\Delta_{\rm LR}^{\epsilon}(T).
\]
For the induction step of $|{\rm Int}(T)|\ge 1$, the tree $T$ is non-trivial
and progressive.  Hence there exists a unique $U\in\Y$ such that
\[
T=B^{r}(U)=U\vee |.
\]
By the defining condition for $\widetilde{\Delta}$ and by
\meqref{eq:uniqblone},
\begin{align}
\widetilde{\Delta}(T)
&=
U\otimes |
+
(\id\otimes B^{r})\widetilde{\Delta}(U),
\mlabel{eq:uniq-step-tilde}
\\
\Delta_{\rm LR}^{\epsilon}(T)
&=
U\otimes |
+
(\id\otimes B^{r})\Delta_{\rm LR}^{\epsilon}(U).
\mlabel{eq:uniq-step-delta}
\end{align}
Write the unique $\backslash$-factorization of $U$ as
\[
U=U_1\backslash\cdots\backslash U_m,
\,\text{ where }\, U_i\in{\rm P}(\Y)\,\text{ and }\, m\ge 0.
\]
Here we employ the convention that $U=|$ if $m=0$.  
For $m\ge 1$, one has
$|{\rm Int}(U_i)|<|{\rm Int}(T)|$, and hence the induction hypothesis gives
\[
\widetilde{\Delta}(U_i)
=
\Delta_{\rm LR}^{\epsilon}(U_i),
\quad \forall\, i=1,\ldots,m.
\]
Since both coproducts satisfy the same derivation rule, the factorization
$U=U_1\backslash\cdots\backslash U_m$ gives
\begin{align*}
\widetilde{\Delta}(U)
&=
\sum_{j=1}^{m}
(U_1\backslash\cdots\backslash U_{j-1})
\backslash \widetilde{\Delta}(U_j)
\backslash
(U_{j+1}\backslash\cdots\backslash U_m)
\\
&=
\sum_{j=1}^{m}
(U_1\backslash\cdots\backslash U_{j-1})
\backslash \Delta_{\rm LR}^{\epsilon}(U_j)
\backslash
(U_{j+1}\backslash\cdots\backslash U_m)
\hspace{1cm}\text{(by the inductive hypothesis)}
\\
&=
\Delta_{\rm LR}^{\epsilon}(U),
\end{align*}
where the sum is understood to be zero for $m=0$.  Hence \meqref{eq:uniq-step-tilde} and \meqref{eq:uniq-step-delta} imply $\widetilde{\Delta}(T)=\Delta_{\rm LR}^{\epsilon}(T).$ This completes the induction and hence the proof.
\end{proof}

We now construct the algebra isomorphism which realizes the correspondence
between forest concatenation and the under product $\backslash$.

\begin{prop}
There exists a unique isomorphism of unitary associative algebras
\begin{equation}
\Theta:
(\bk\mathcal F,\cdot,\funit)
\longrightarrow
(H_{\rm LR}^{\epsilon},\backslash,|)
\mlabel{isomorp}
\end{equation}
such that
\[
\Theta\bigl(B^+(F)\bigr)
=
B^{r}\bigl(\Theta(F)\bigr)
=
\Theta(F)\vee |,
\quad \forall \, F\in\mathcal F.
\]
\mlabel{prop:theta-alg-iso}
\end{prop}

\begin{proof}
We first construct
\[
\theta:\mathcal T\longrightarrow {\rm P}(\Y).
\]
Every $t\in\mathcal T$ can be uniquely written as $t=B^+(F)$ with
$F=t_1\cdots t_m\in\mathcal F$ and $m\geq 0$. Define $\theta$ recursively by
\begin{equation}
\theta\bigl(B^+(F)\bigr)
:=
\bigl(\theta(t_1)\backslash\cdots\backslash\theta(t_m)\bigr)\vee |,
\mlabel{eq:def-theta-tree}
\end{equation}
where the empty $\backslash$-product is understood to be $|$.

We complete the proof in three steps: first proving that $\theta$ is
surjective, then injective, and finally extending it from rooted trees to
rooted forests.

\noindent\textbf{Step 1.} (Surjectivity of $\theta$).  Let $T'\in{\rm P}(\Y)$ and set
$k=|{\rm Int}(T')|$.  We argue by induction on $k\geq 1$.
For the initial step of $k=1$, we have
\[
T'=|\vee|=\theta(B^+(\funit)).
\]
For the induction step of $k\ge 2$, write $T'=u\vee |$.  Write the unique $\backslash$-factorization of $u$ into progressive trees as
\[
u=U_1\backslash\cdots\backslash U_m.
\]
Here $U_i\in{\rm P}(\Y)$ for $i=1,\ldots,m$,
with the convention that the empty product is $|$.  Since
$|{\rm Int}(U_i)|<|{\rm Int}(T')|$, the induction hypothesis gives
$t_i\in\mathcal T$ such that $\theta(t_i)=U_i$.  For
$F=t_1\cdots t_m$,  \meqref{eq:def-theta-tree} gives
\[
\theta(B^+(F))
=
\bigl(\theta(t_1)\backslash\cdots\backslash\theta(t_m)\bigr)\vee |
=
u\vee |
=
T'.
\]
Thus $\theta$ is surjective.

\noindent\textbf{Step 2.} (Injectivity of $\theta$).  First note the size identity
\begin{equation}
|{\rm Int}(\theta(t))|=|t|,
\quad \forall \, t\in\mathcal T,
\mlabel{eq:theta-size}
\end{equation}
where $|t|$ denotes the number of vertices of $t$.  Indeed, if
$t=B^+(t_1\cdots t_m)$, then induction gives
\[
|{\rm Int}(\theta(t))|
=
1+\sum_{i=1}^m |{\rm Int}(\theta(t_i))|
=
1+\sum_{i=1}^m |t_i|
=
|t|.
\]
Now suppose
$\theta(B^+(F))=\theta(B^+(G)).$
By \meqref{eq:theta-size}, the rooted trees $B^+(F)$ and $B^+(G)$ have the
same number of vertices.  We prove by induction on
\[
N:=|V(B^+(F))|=|V(B^+(G))|
\]
that this equality implies $B^+(F)=B^+(G)$.
For the initial step of $N=1$, both $B^+(F)$ and $B^+(G)$ consist only of the
root vertex.  Hence $F=G=\funit$, and therefore $B^+(F)=B^+(G)$.

For the induction step of $N\ge 2$, assume that the assertion holds for rooted
trees with fewer than $N$ vertices.  Write
\[
F=t_1\cdots t_m,
\quad
G=s_1\cdots s_n.
\]
Using \meqref{eq:def-theta-tree}, the equality
$\theta(B^+(F))=\theta(B^+(G))$ becomes
\[
\bigl(\theta(t_1)\backslash\cdots\backslash\theta(t_m)\bigr)\vee |
=
\bigl(\theta(s_1)\backslash\cdots\backslash\theta(s_n)\bigr)\vee |.
\]
Hence
\[
\theta(t_1)\backslash\cdots\backslash\theta(t_m)
=
\theta(s_1)\backslash\cdots\backslash\theta(s_n).
\]
By uniqueness of the $\backslash$-factorization into progressive trees,
\[
m=n,
\quad
\theta(t_i)=\theta(s_i),
\quad \forall \, i=1,\ldots,m.
\]
Applying the
induction hypothesis gives $t_i=s_i$ for all $i$.  Thus $F=G$, and hence
$B^+(F)=B^+(G)$.  Therefore $\theta$ is injective.

\noindent\textbf{Step 3.} (From $\theta$ to $\Theta$).  Since every forest factors uniquely into rooted trees and every planar binary
tree factors uniquely into progressive trees, the bijection $\theta$ induces a
bijection
$\Theta:\mathcal F\longrightarrow \Y$
given by
\[
\funit\longmapsto |,
\quad
t_1\cdots t_m
\longmapsto
\theta(t_1)\backslash\cdots\backslash\theta(t_m).
\]
Then
\[
\Theta(FG)=\Theta(F)\backslash\Theta(G),
\quad \forall \, F,G\in\mathcal F.
\]
Moreover, if $F=t_1\cdots t_m$, then
\[
\Theta(B^+(F))
=
\bigl(\theta(t_1)\backslash\cdots\backslash\theta(t_m)\bigr)\vee |
=
\Theta(F)\vee |
=
B^{r}(\Theta(F)).
\]
Thus $\Theta$ induces the required isomorphism \meqref{isomorp} of unitary
associative algebras. 

Its uniqueness follows from the prescribed value on $\funit$, the identity
$\Theta(B^+(F))=B^{r}(\Theta(F))$, and multiplicativity with respect to
concatenation and $\backslash$.
\end{proof}

We now combine the uniqueness characterization with the algebra isomorphism
$\Theta$ to identify the transported rooted-forest coproduct.

\begin{theorem}
Let $\Theta:\bk\mathcal F\to H_{\rm LR}^{\epsilon}$
be the unitary algebra isomorphism of Proposition~\mref{prop:theta-alg-iso}.
Then
\[
\Theta:
(\bk\mathcal F,\cdot,\funit,\Delta_{\rm RT}^{\epsilon})
\longrightarrow
(H_{\rm LR}^{\epsilon},\backslash,|,\Delta_{\rm LR}^{\epsilon})
\]
is an isomorphism of unitary infinitesimal bialgebras.
\mlabel{thm:bridge}
\end{theorem}

\begin{proof}
We  only need to prove
\begin{equation}
(\Theta\otimes \Theta)\Delta_{\rm RT}^{\epsilon}(F)
=
\Delta_{\rm LR}^{\epsilon}\bigl(\Theta(F)\bigr),
\quad \forall\, F\in \bk\mathcal F.
\mlabel{eq:bridge}
\end{equation}
Define
\[
\widetilde{\Delta}
:=
(\Theta\otimes \Theta)\circ \Delta_{\rm RT}^{\epsilon}\circ \Theta^{-1}
:
H_{\rm LR}^{\epsilon}\longrightarrow
H_{\rm LR}^{\epsilon}\otimes H_{\rm LR}^{\epsilon}.
\]
We verify the three conditions in Proposition~\mref{prop:uniqbl}.
First, since $\Theta(\funit)=|$ and
$\Delta_{\rm RT}^{\epsilon}(\funit)=0$, we have
$\widetilde{\Delta}(|)=0.$
Next, let $T\in\Y$ and write $T=\Theta(F)$ with $F\in\mathcal F$.  Since
\[
B^{r}(T)
=
B^{r}(\Theta(F))
=
\Theta(B^+(F)),
\]
we get
\begin{align*}
\widetilde{\Delta}\bigl(B^{r}(T)\bigr)
&=
(\Theta\otimes \Theta)
\Delta_{\rm RT}^{\epsilon}\bigl(B^+(F)\bigr)
\\
&\overset{\eqref{eq:rt-cocycle}}=
(\Theta\otimes \Theta)
\Bigl(
F\otimes \funit
+
(\id\otimes B^+)\Delta_{\rm RT}^{\epsilon}(F)
\Bigr)
\\
&=
\Theta(F)\otimes |
+
(\id\otimes B^{r})
(\Theta\otimes \Theta)\Delta_{\rm RT}^{\epsilon}(F)
\\
&=
T\otimes |
+
(\id\otimes B^{r})\widetilde{\Delta}(T).
\end{align*}
Finally, we check the derivation rule.  Let $S,T\in\Y$ and write
\[
S=\Theta(F),
\quad
T=\Theta(G)
\]
for some $F,G\in\mathcal F$.
Then $S\backslash T=\Theta(FG)$.  For
$X=\sum_a X'_a\otimes X''_a\in
\bk\mathcal F\otimes\bk\mathcal F,$
the algebra isomorphism $\Theta$ gives
\[
(\Theta\otimes\Theta)(F\cdot X)
=
\Theta(F)\backslash(\Theta\otimes\Theta)(X),
\quad
(\Theta\otimes\Theta)(X\cdot G)
=
(\Theta\otimes\Theta)(X)\backslash\Theta(G).
\]
Hence
\begin{align*}
\widetilde{\Delta}(S\backslash T)
&=
(\Theta\otimes \Theta)\Delta_{\rm RT}^{\epsilon}(FG)
\\
&\overset{\eqref{eq:rt-derivation}}=
(\Theta\otimes \Theta)
\Bigl(
F\cdot \Delta_{\rm RT}^{\epsilon}(G)
+
\Delta_{\rm RT}^{\epsilon}(F)\cdot G
\Bigr)
\\
&=
S\backslash \widetilde{\Delta}(T)
+
\widetilde{\Delta}(S)\backslash T.
\end{align*}
Thus $\widetilde{\Delta}$ satisfies the three conditions of
Proposition~\mref{prop:uniqbl}.  Therefore
$\widetilde{\Delta}=\Delta_{\rm LR}^{\epsilon},$
and
\[
(\Theta\otimes \Theta)\Delta_{\rm RT}^{\epsilon}(F)
=
\Delta_{\rm LR}^{\epsilon}\bigl(\Theta(F)\bigr),
\quad \forall \, F\in\bk\mathcal F.
\]
This completes the proof.
\end{proof}

Let us expose an example for better understanding of the above isomorphism. 

\begin{exam}
\mlabel{ex:bridge-typical}
Let
$\bullet:=B^+(\funit), t:=B^+\bigl(\bullet\,B^+(\bullet)\bigr).
$
Then $t$ is the planar rooted tree
\[
\begin{tikzpicture}[baseline=-2pt, scale=0.9]
\coordinate (w) at (0,0);
\coordinate (b) at (-0.75,0.75);
\coordinate (a) at (0.75,0.75);
\coordinate (x) at (0.75,1.55);
\draw (w) -- (b);
\draw (w) -- (a);
\draw (a) -- (x);
\node[inner sep=0pt] at (w) {$\bullet$};
\node[inner sep=0pt] at (b) {$\bullet$};
\node[inner sep=0pt] at (a) {$\bullet$};
\node[inner sep=0pt] at (x) {$\bullet$};
\end{tikzpicture}.
\]
Under the map $\Theta$, we have
$$\Theta(\bullet)=\YY{}$$ and
$$\Theta\bigl(B^+(\bullet)\bigr)
=B^{r}\bigl(\Theta(\bullet)\bigr)=\YY{\yy10}.$$
Moreover,
\begin{align*}
\Theta\bigl(\bullet\,B^+(\bullet)\bigr)
&=
\Theta(\bullet)\backslash \Theta\bigl(B^+(\bullet)\bigr)
=
\YY{}\backslash \YY{\yy10}
=
\YY{\yy11\yy22},
\\
\Theta(t)
&=
B^{r}\bigl(\Theta(\bullet\,B^+(\bullet))\bigr)
=
\bigl(\YY{}\backslash \YY{\yy10}\bigr)\vee |
=
\YY{\yy10\yy21\yy32}.
\end{align*}
On the rooted-forest side,
\[
\Delta_{\rm RT}^{\epsilon}(\bullet)
=
\funit\otimes \funit,
\quad
\Delta_{\rm RT}^{\epsilon}\bigl(B^+(\bullet)\bigr)
=
\bullet\otimes \funit+\funit\otimes \bullet.
\]
Hence
\begin{align*}
\Delta_{\rm RT}^{\epsilon}\bigl(\bullet\,B^+(\bullet)\bigr)
&\overset{\eqref{eq:rt-derivation}}=
\bullet\cdot
\Delta_{\rm RT}^{\epsilon}\bigl(B^+(\bullet)\bigr)
+
\Delta_{\rm RT}^{\epsilon}(\bullet)\cdot B^+(\bullet)
\\
&=
\bullet\bullet\otimes \funit
+
\bullet\otimes \bullet
+
\funit\otimes B^+(\bullet),
\end{align*}
and
\begin{align*}
\Delta_{\rm RT}^{\epsilon}(B^+\bigl(\bullet\,B^+(\bullet)\bigr))
&\overset{\eqref{eq:rt-cocycle}}=
\bullet\,B^+(\bullet)\otimes \funit
+
(\id\otimes B^+)
\Delta_{\rm RT}^{\epsilon}\bigl(\bullet\,B^+(\bullet)\bigr)
\\
&=
\bullet\,B^+(\bullet)\otimes \funit
+
\bullet\bullet\otimes \bullet
+
\bullet\otimes B^+(\bullet)
+
\funit\otimes B^+\bigl(B^+(\bullet)\bigr).
\end{align*}
Applying \meqref{eq:bridge}, we obtain
\begin{align*}
\Delta_{\rm LR}^{\epsilon}\Bigl(\Theta\Bigl(
\vcenter{\hbox{
\begin{tikzpicture}[baseline=(current bounding box.center), scale=0.9]
\coordinate (w) at (0,0);
\coordinate (b) at (-0.75,0.75);
\coordinate (a) at (0.75,0.75);
\coordinate (x) at (0.75,1.55);
\draw (w) -- (b);
\draw (w) -- (a);
\draw (a) -- (x);
\node[inner sep=0pt] at (w) {$\bullet$};
\node[inner sep=0pt] at (b) {$\bullet$};
\node[inner sep=0pt] at (a) {$\bullet$};
\node[inner sep=0pt] at (x) {$\bullet$};
\end{tikzpicture}
}}
\Bigr)\Bigr)
&=
(\Theta\otimes\Theta)\Delta_{\rm RT}^{\epsilon}\Bigl(
\vcenter{\hbox{
\begin{tikzpicture}[baseline=(current bounding box.center), scale=0.9]
\coordinate (w) at (0,0);
\coordinate (b) at (-0.75,0.75);
\coordinate (a) at (0.75,0.75);
\coordinate (x) at (0.75,1.55);
\draw (w) -- (b);
\draw (w) -- (a);
\draw (a) -- (x);
\node[inner sep=0pt] at (w) {$\bullet$};
\node[inner sep=0pt] at (b) {$\bullet$};
\node[inner sep=0pt] at (a) {$\bullet$};
\node[inner sep=0pt] at (x) {$\bullet$};
\end{tikzpicture}
}}
\Bigr)
\\
&=
\bigl(\YY{}\backslash \YY{\yy10}\bigr)\otimes |
+
\bigl(\YY{}\backslash \YY{}\bigr)\otimes \YY{}
+
\YY{}\otimes \YY{\yy10}
+
|\otimes \bigl(\YY{\yy10}\vee |\bigr)
\\
&=
\YY{\yy11\yy22}\otimes |
+
\YY{\yy11}\otimes \YY{}
+
\YY{}\otimes \YY{\yy10}
+
|\otimes \YY{\yy10\yy20}.
\end{align*}
\end{exam}

We conclude this section  with a final remark.

\begin{remark}
\begin{enumerate}
\item The isomorphism $\Theta$ in Theorem~\mref{thm:bridge} has the same underlying set map
as the Aguiar--Sottile rooted-forest/binary-tree bijection~\mcite{Aguiar2006},
which is the map
$
\varphi:\mathcal F\longrightarrow \Y
$
defined by
\[
\varphi(\funit)=|,
\quad
\varphi(FG)=\varphi(F)\backslash\varphi(G),
\quad
\varphi(B^+(F))=\varphi(F)\vee |.
\]
Indeed, Proposition~\mref{prop:theta-alg-iso} shows that $\Theta$ satisfies the same
rules on $F\in\mathcal F$.  The algebraic structures studied here are
different.

\item Since Aguiar--Sottile's framework also connects planar binary trees with
permutations, it is natural to ask whether one can study an infinitesimal
bialgebra structure on permutations compatible with these maps.
\end{enumerate}
\end{remark}

\smallskip
\noindent {\bf Acknowledgments.} This work is supported by the National Natural Science Foundation of China (12571019), the Natural Science Foundation of Gansu Province (25JRRA644) and Innovative Fundamental Research Group Project of Gansu Province (23JRRA684).

\end{document}